\documentclass[12pt]{amsart}

\usepackage[mathscr]{eucal}
\usepackage{amscd}
\usepackage{amsfonts}
\usepackage{amsmath, amsthm, amssymb}
\usepackage{latexsym}
\usepackage[dvips]{graphics}

\theoremstyle{plain} 
\newtheorem{thm}{Theorem}

\newtheorem{prop}[thm]{Proposition}

\theoremstyle{definition}
\newtheorem{defn}[thm]{Definition}
\newtheorem{ex}[thm]{Example}

\theoremstyle{remark}
\newtheorem{rem}[thm]{Remark}

\numberwithin{equation}{section}

\title{INTRODUCTION TO QUIVER VARIETIES \\
 --- FOR RING AND REPRESENTATION THEORISTS   
}

\author{Hiraku Nakajima} 

\address{
\begin{flushleft}
        \hspace{0.3cm}  Research Institute for Mathematical Sciences \\
         \hspace{0.3cm}  Kyoto University \\
         \hspace{0.3cm}  Kyoto, Kyoto 606-8502 JAPAN\\
\end{flushleft}
}
\email{nakajima@kurims.kyoto-u.ac.jp} 

\thanks{The paper is in a final form and no version of it will
be submitted for publication elsewhere.}

\newcommand{\vout}[1]{\operatorname{o}(#1)}
\newcommand{\vin}[1]{\operatorname{i}(#1)}
\newcommand{\barinv}{
  \setbox5=\hbox{A}\overline{\rule{0mm}{\ht5}\hspace*{\wd5}}\,}
\newcommand{\End}{\operatorname{End}}
\newcommand{\Hom}{\operatorname{Hom}}
\newcommand{\Ker}{\operatorname{Ker}}
\newcommand{\Ima}{\operatorname{Im}}
\newcommand{\GL}{\operatorname{GL}}
\newcommand{\ZZ}{{\mathbb Z}}
\newcommand{\CC}{{\mathbb C}}
\newcommand{\RR}{{\mathbb R}}
\newcommand{\tr}{\operatorname{tr}}
\newcommand{\dslash}{/\!\!/}
\newcommand{\fM}{{\mathfrak M}}
\newcommand{\fC}{{\mathfrak C}}
\newcommand{\La}{{\mathfrak L}}

\theoremstyle{remark}
\newtheorem{Exercise}[thm]{Exercise}

\usepackage[all,cmtip]{xy}

\begin{document}

\maketitle


\begin{abstract}


    We review the definition of quiver varieties and their relation to
    representation theory of Kac-Moody Lie algebras. Target readers
    are ring and representation theorists. We emphasize important
    roles of first extension groups of the preprojective algebra
    associated with a quiver.




\end{abstract}


\section{Introduction}

This is a review on quiver varieties written for the proceeding of
49th Symposium on Ring Theory and Representation Theory at Osaka
Prefecture University, 2016 Summer, based on my two lectures. 
Quiver varieties are spaces parametrising representations of
preprojective algebras associated with a quiver, hence they are
closely related to Ring Theory and Representation Theory. This is the
reason why I was invited to give lectures, even though my main
research interest is geometric representation theory.

The purpose of this review is to explain the definition of quiver
varieties and the main result in \cite{Nakajima:Na94,Nakajima:Na98},
which is 20 years old. Why do I write a review of such an old result~?
There exist several reviews of quiver varieties already. An earlier
review with the same target readers is \cite{Nakajima:proc}.
There are other reviews \cite{Nakajima:Schiffmann,Nakajima:Gi}, and
also a book \cite{Nakajima:Ki16}.

Besides shortest among existing reviews, this one has a special
feature: I put emphasis on the complex \eqref{eq:8}, which has been
used at various places in the theory of quiver varieties.
It is a familiar complex in representation theory, as it computes
homomorphisms and $\operatorname{Ext}^1$ between (framed)
representations of the preprojective algebra.

Importance of $\operatorname{Ext}^1$ is clear to ring and
representation theorists. A purpose of this review is to explain its
importance for quiver varieties.
Quiver varieties themselves could be loosely viewed as
\emph{nonlinear} analog of self-$\operatorname{Ext}^1$ as their
tangent spaces are nothing but $\operatorname{Ext}^1$ of modules with
themselves. (See at the end of \S\ref{subsec:2.1}.)
A particularly nice feature of \eqref{eq:8} for the case of tangent
spaces is the vanishing of the first and third cohomology groups.
This follows from the stability condition used in the definition of
quiver varieties.
It implies the smoothness of quiver varieties.

The complex \eqref{eq:8} where one of representation corresponds to a
simple representation $S_i$ for a vertex $i$ is also important. See
\S\ref{subsec:definition-kac-moody}. The stability condition implies
that the first cohomology group vanishes, hence the difference of
dimensions of the second and third cohomology groups is the Euler
characteristic of the complex. This simple observation plays an
important role.
The complex also appears in a definition of Kashiwara crystal
structure on the set of irreducible components of lagrangian
subvarieties in quiver varieties. See \S\ref{subsec:crystal}.

I will not list earlier references which I studied before writing
\cite{Nakajima:Na94,Nakajima:Na98}. So readers who pay attention on
history should read original papers. But two papers
\cite{Nakajima:Ringel,Nakajima:MR1035415} were so fundamental, let me
recall how I encountered them.
In 1989 Summer I introduced moduli spaces $\fM(V,W)$ of (framed)
representations of the preprojective algebra of an affine $ADE$ quiver
$Q = (Q_0,Q_1)$ with dimension vector $V$ and framing vector $W$ with
Kronheimer \cite{Nakajima:KN}. This construction had an origin in the
gauge theory.
Hence I thought that they are important spaces, and I was interested
in symplectic geometry, topology, etc, of $\fM(V,W)$ as a geometer.

In 1990 Summer I heard Lusztig's plenary talk at ICM Kyoto, explaing
his construction \cite{Nakajima:MR1035415} of the canonical base of
the upper triangular subalgebra $\mathbf U^-$ of the quantized
enveloping algebra $\mathbf U = \mathbf U_q(\mathfrak g)$, built on an
earlier result by Ringel \cite{Nakajima:Ringel}.

Since Ringel and Lusztig's constructions were based on fields which
were not familiar to me at that time, it took several years until I
realized that the direct sum of homology groups of $\fM(V,W)$ is a
representation of the Kac-Moody Lie algebra $\mathfrak g$ associated
with the quiver $Q$, as variants of their construction
\cite{Nakajima:Na94,Nakajima:Na98}. This makes sense for any quiver,
hence I named $\fM(V,W)$ \emph{quiver varieties}, and started to study
further structures of $\fM(V,W)$. Rather unexpectedly, quiver
varieties have lots of structures, and they are still actively studied
by various people even now.

\subsection*{Acknowledgment}

I thank the organizers of the symposium for invitation.

\section{Notation and basic definitions}

\subsection{Preprojective algeras and extension
  groups}\label{subsec:2.1}

Let $Q = (Q_0,Q_1)$ be a quiver, where $Q_0$ is the set of vertices,
and $Q_1$ is the set of oriented edges. We always assume $Q$ is
finite.
Let $\vout{h}$, $\vin{h}$ denote the outgoing and incoming vertices of
an edge $h$. For $h\in Q_1$, we consider an edge with opposite
orientation and denote it by $\overline{h}$, hence
$\vout{\overline{h}} = \vin{h}$, $\vin{\overline{h}} = \vout{h}$. We
add $\overline{Q}_1 = \{ \overline{h} \mid h\in Q_1\}$ to $Q_1$, and
consider the \emph{doubled quiver}
$Q^{\rm dbl} = (Q_0, Q_1\sqcup\overline{Q}_1)$. We denote
$Q_1\sqcup\overline{Q}_1$ by $Q^{\rm dbl}_1$.  We extend
$\barinv\colon Q_1\to\overline{Q}_1$ to $Q^{\rm dbl}_1$ so that
$\overline{\overline h} = h$.

Let $V = \bigoplus_{i\in Q_0} V_i$ be a finite dimensional complex
$Q_0$-graded vector space. Its dimension vector
$(\dim V_i)_{i\in Q_0} \in \ZZ_{\ge 0}^{Q_0}$ is denoted simply by
$\dim V$.
We introduce a vector space
\begin{equation*}
  \mathbf N(V) = \bigoplus_{h\in Q_1} \Hom(V_{\vout{h}}, V_{\vin{h}}), \qquad
  G_V = \prod_{i\in Q_0} \GL(V_i).
\end{equation*}
An element in $\mathbf N(V)$ is denoted by
$B = \bigoplus_{h\in Q_1} B_h$ or $(B_h)_{h\in Q_1}$, where $B_h$ is
the component in $\Hom(V_{\vout{h}}, V_{\vin{h}})$. Similarly an
element in $G_V$ is denoted by $g = \prod_{i\in Q_0} g_i = (g_i)_{i\in Q_0}$.

We have a set-theoretical bijection
\begin{equation*}
  \left\{\begin{minipage}{.7\linewidth}
    isomorphism classes of representations of $Q$ whose dimension vector
    is $\dim V$
  \end{minipage}\right\}
  \longleftrightarrow \mathbf N(V)/G_V.
\end{equation*}
Here $G_V$ acts on $\mathbf N(V)$ by conjugation. By abuse of
terminology a point $B\in\mathbf N(V)$ is often called a
representation of $Q$.

We consider the cotangent space to $\mathbf N(V)$:
\begin{equation*}
  \mathbf M(V) = \mathbf N(V)\oplus \mathbf N(V)^*
  = \bigoplus_{h\in Q_1^{\rm dbl}} \Hom(V_{\vout{h}}, V_{\vin{h}}).
\end{equation*}
The action of $G_V$ on $\mathbf M(V)$ is defined also by
conjugation. It preserves the symplectic form on $\mathbf M(V)$ given
by the natural pairing. We consider the associated \emph{moment map}
\begin{equation*}
  \mu\colon \mathbf M(V) \to \bigoplus_{i\in Q_0} \operatorname{End}(V_i);
  \bigoplus_{h\in Q^{\rm dbl}_1} B_h \mapsto \bigoplus_{i\in Q_0} \sum_{\substack{h\in Q^{\rm dbl}_1 \\ \vin{h}=i}}
    \varepsilon(h) B_h B_{\overline{h}},
\end{equation*}
where $\varepsilon(h) = 1$ if $h\in Q_1$, $-1$ if
$h\in\overline{Q}_1$.

Note that $\bigoplus_i \End(V_i)$ is the Lie algebra of the group $G_V$.

The moment map has its origin in symplectic geometry: the quotient
space $\mu^{-1}(0)/G_V$ is the cotangent bundle
$T^*(\mathbf N(V)/G_V)$ of $\mathbf N(V)/G_V$:
$(B_h)_{h\in\overline{Q}_1}$ is a cotangent vector, and the equation
$\mu=0$ means it vanishes to the tangent direction to $G_V$-orbits:
\begin{multline*}
  (B_{\overline h})_{\overline{h}\in\overline{Q}_1} \perp T\left( G_V\cdot (B_h)_{h\in Q_1}\right) \\
  \Longleftrightarrow
  \sum_{h\in Q_1} \tr (B_{\overline{h}} (\xi_{\vin{h}} B_h - B_h \xi_{\vout{h}})) = 0
  \quad \forall (\xi_i)_i \in \bigoplus_{i\in Q_0} \End(V_i) \\
  \Longleftrightarrow
  \sum_{h\in Q_1, \vin{h} = i} B_h B_{\overline{h}} -
  \sum_{h\in Q_1, \vout{h}=i} B_{\overline{h}} B_h =0.
\end{multline*}
But this must be understood with care, as $\mathbf N(V)/G_V$ is not a
manifold, nor even a Hausdorff space, in general. In the next section
we introduce a modification of the quotient $\mu^{-1}(0)/G_V$ (we also
add framing), which is a smooth algebraic variety. But it is
\emph{not} a cotangent bundle, nor a vector bundle over another
manifold. It is because a similar modification of $\mathbf N(V)/G_V$
is usually smaller or quite often $\emptyset$, and its cotangent
bundle is an open subset of $\mu^{-1}(0)/G_V$. Here the cotangent
bundle of an empty set is understood as an empty set.

Let us also note that $\mu=0$ is the defining relation of the
\emph{preprojective algebra} $\Pi(Q)$, introduced by
Gelfand-Ponomarev, and further studied by Dlab-Ringel. Again by abuse
of terminology, a point $B$ in $\mu^{-1}(0)$ is often called a 
representation of the preprojective algebra associated with
$Q^{\rm dbl}$ (or $Q$).

Let us explain another related interpretation of $\mu$. Let us take a
point $B\in\mathbf M(V)$ and consider
\begin{equation}\label{eq:3}
\begin{gathered}
  \bigoplus_{i\in Q_0} \End(V_i) \xrightarrow{\iota}
  \mathbf M(V) \xrightarrow{d\mu}
  \bigoplus_{i\in Q_0} \End(V_i),
  \\
  \iota(\xi) = (\xi_{\vin{h}} B_h - B_h \xi_{\vout{h}})_h,
  \quad
  d\mu_i(C) = \sum_{\substack{h\in Q_1^{\rm dbl} \\ \vin{h} = i}}
  \varepsilon(h) (B_h C_{\overline{h}} + C_h B_{\overline{h}}).
\end{gathered}
\end{equation}
The linear map $\iota$ is nothing but the differential of the
$G_V$-action given by conjugation, when we undertand
$\bigoplus \End(V_i)$ as the Lie algebra of $G_V$. On the other hand,
$d\mu$ is the differential of the moment map $\mu$. Note also that
this is a complex, i.e., $d\mu\circ\iota = 0$ if $\mu(B) = 0$.

Now we observe
\begin{equation}
  \label{eq:1}
  \begin{minipage}{.9\linewidth}
    $d\mu$ is the transpose of $\iota$ when we identify
    $\mathbf M(V)$ with its dual space via the symplectic form.
  \end{minipage}
\end{equation}
($\End(V_i)$ is self-dual by the trace pairing.)
Hence we have
\begin{equation}\label{eq:4}
  \Ker \iota \cong \left(\operatorname{Cok}d\mu\right)^\vee, \qquad
  \Ker d\mu \cong \left(\operatorname{Cok}\iota\right)^\vee.
\end{equation}

As we mentioned above, we will consider a modification of a quotient
space $\mu^{-1}(0)/G_V$ later, which is a smooth algebraic
variety. Let us omit the detail at this stage, and assume
$\mu^{-1}(0)/G_V$ is smooth so that the quotient map
$\mu^{-1}(0)\to \mu^{-1}(0)/G_V$ is a submersion. Then the tangent
space of $\mu^{-1}(0)/G_V$ at $[B]$ is given by
\begin{equation*}
    T_{[B]}(\mu^{-1}(0)/G_V) = \Ker d\mu/\Ima \iota.
\end{equation*}
Here $[B]$ denotes the point in $\mu^{-1}(0)/G_V$ given by
$B\in\mu^{-1}(0)$. From the observation \eqref{eq:4} above, the right
hand side has the induced symplectic form. Let us denote it
$\omega$. We consider $\omega$ as a differential form on the manifold
$\mu^{-1}(0)/G_V$. Let us check that $\omega$ is closed, i.e.,
$d\omega = 0$. In fact, it is enough to check that the pull-back of
$\omega$ to $\mu^{-1}(0)$ is closed as the quotient map is a
submersion. By the definition, the pull-back is nothing but the
restriction of the symplectic form on $\mathbf M(V)$. Then as $d$
commutes with the restriction, the closedness of the pull-back follows
from that of the symplectic form on $\mathbf M(V)$. But the latter is
trivial as $\mathbf M(V)$ is a vector space and its symplectic form is
constant.

Let us note that the complex \eqref{eq:3} can be modified to one
associated with a pair $B^1\in \mathbf M(V^1)$,
$B^2\in \mathbf M(V^2)$ where both satisfy $\mu=0$:
\begin{equation}\label{eq:5}
\begin{gathered}
  \bigoplus_{i\in Q_0} \Hom(V_i^1,V^2_i) \xrightarrow{\alpha}
  \bigoplus_{h\in Q_1^{\rm dbl}} \Hom(V^1_{\vout{h}}, V^2_{\vin{h}})
  \xrightarrow{\beta}
  \bigoplus_{i\in Q_0} \Hom(V_i^1,V^2_i)
  \\
  \alpha(\xi) = (\xi_{\vin{h}} B^1_h - B^2_h \xi_{\vout{h}})_h,
  \\
  \beta(C,D,E) = \sum_{\substack{h\in Q_1^{\rm dbl} \\ \vin{h} = i}}
  \varepsilon(h) (B^2_h C_{\overline{h}} + C_h B^1_{\overline{h}}).
\end{gathered}
\end{equation}

This complex is important in the representation theory of
preprojective algebras. Let us regard $B^1$, $B^2$ as modules of the
preprojective algebra $\Pi(Q)$. Then we have
\begin{equation*}
    \begin{gathered}
    \Ker\alpha \cong \Hom_{\Pi(Q)}(B^1,B^2), \qquad
    \operatorname{Coker}\beta \cong \Hom_{\Pi(Q)}(B^2,B^1)^\vee,\\
    \Ker\beta/\Ima\alpha \cong \operatorname{Ext}^1_{\Pi(Q)}(B^1,B^2).
    \end{gathered}
\end{equation*}
The first two isomorphisms are just by definition and the computation
of the transpose of $\beta$ as above. The last isomorphism is proved
in \cite{Nakajima:CB2}.

From this observation, the quotient space $\mu^{-1}(0)/G_V$ is a
\emph{nonlinear} version of the self-extension
$\operatorname{Ext}^1_{\Pi(Q)}(B,B)$, as a tangent space is linear
approximation of a manifold. This partly explains importance of study
of $\mu^{-1}(0)/G_V$, as $\operatorname{Ext}^1$ is a fundamental
object in representation theory. It is also \emph{deeper} than
$\operatorname{Ext}^1$, as the tangent space only reflects a local
structure of the manifold, and cannot see global structures, such as
topology of the manifold.

\subsection{Framed representations}

Now we take an additional $Q_0$-graded finite-dimensional complex
vector space $W = \bigoplus_{i\in Q_0} W_i$ and introduce
\begin{equation*}
  \mathbf M(V,W) = \bigoplus_{h\in Q_1^{\rm dbl}} \Hom(V_{\vout{h}}, V_{\vin{h}})
  \oplus \bigoplus_{i\in Q_0} \Hom(W_i, V_i)
  \oplus
  \Hom(V_i,W_i).
\end{equation*}
An element of the additional factor
$\bigoplus_{i\in Q_0} \Hom(W_i, V_i) \oplus \Hom(V_i,W_i)$ is called a
\emph{framing} of a quiver representation, and we denote it by
$I = (I_i)_{i\in Q_0}$, $J = (J_i)_{i\in Q_0}$. A point
$(B,I,J)\in\mathbf M(V,W)$ is called a \emph{framed representation}.

We have an action of $G_V$, and also of
$G_W = \prod_{i\in Q_0}\GL(W_i)$, on $\mathbf M(V,W)$ by
conjugation. It will be important for various applications of quiver
varieties, but we will \emph{not} use the latter action in this
review.

We have the moment map
\begin{gather*}
  \mu = (\mu_i)_i \colon\mathbf M(V,W)\to \bigoplus_i \mathfrak{gl}(V_i);
  \quad
  \mu_i(B, I, J) = \sum_{\substack{h\in Q_1^{\rm dbl}
         \\ \vin{h}=i}} \varepsilon(h) B_h B_{\overline{h}} + I_i J_i,
\end{gather*}
as above.

\begin{rem}\label{rem:CB}
  The framing factor naturally appeared in \cite{Nakajima:KN}, and it
  is also important in applications of quiver varieties to
  representation theory of Lie algebras, as $\dim W$ will be
  identified with a highest weight of a representation.
  However the author could not find earlier appearances in quiver
  representation literature, and he gave an explanation for the ring
  and representation theory community in \cite{Nakajima:proc}.

  On the other hand, Crawley-Boevey \cite{Nakajima:CB} found the
  following trick, which makes $\mathbf M(V,W)$ as the previous
  $\mathbf M(V')$ with a different quiver and a graded vector space
  $V'$ \cite{Nakajima:CB}: Add a vertex $\infty$ to $Q$, and draw
  edges from $\infty$ to $i\in Q_0$ as many as $\dim W_i$. We then
  defined $V'$ as $V$ plus one-dimensional vector space at the vertex
  $\infty$. We identify $\Hom(W_i, V_i)$ with
  $\Hom(\CC,V_i)^{\oplus \dim W_i}$ after taking a base of
  $W_i$. Hence we have $\mathbf M(V,W) = \mathbf M(V')$.
\end{rem}

\begin{rem}
  In \cite{Nakajima:Kr,Nakajima:KN} a deformation of the equation
  $\mu = 0$ as $\mu_i(B,I,J) = \zeta^\CC_i \operatorname{id}_{V_i}$ is
  considered. Here $\zeta^\CC = (\zeta^\CC_i)_i\in \CC^{Q_0}$. It
  motivated Crawley-Boevey and Holland \cite{Nakajima:CH} to study
  \emph{deformed preprojective algebras}. We restrict our interest
  only on the undeformed case $\zeta^\CC = 0$ in this review, though
  many results remain true for deformed case.
\end{rem}

Observation on symplectic forms for the case no framing $W$ remains
valid in the case with $W$, as $\mathbf M(V,W)$ is still a symplectic
vector space. In particular, we have an induced symplectic form on the
quotient $\mu^{-1}(0)/G_V$.

Let us write down a framed analog of \eqref{eq:5} for a pair
$(B^1,I^1,J^1)\in \mathbf M(V^1,W^1)$,
$(B^2,I^2,J^2)\in \mathbf M(V^2,W^2)$ where both satisfy $\mu=0$:
\begin{equation}\label{eq:8}
\begin{gathered}
  \bigoplus_{i\in Q_0} \Hom(V_i^1,V^2_i) \xrightarrow{\alpha}
  \begin{gathered}[m]
  \bigoplus_{h\in Q_1^{\rm dbl}} \Hom(V^1_{\vout{h}}, V^2_{\vin{h}})
  \\
  \oplus 
  \\
  \bigoplus_{i\in Q_0} \Hom(W^1_i, V^2_i)
  \oplus
  \Hom(V^1_i,W^2_i)
  \end{gathered}
  \xrightarrow{\beta}
  \bigoplus_{i\in Q_0} \Hom(V_i^1,V^2_i)
  \\
  \alpha(\xi) = (\xi_{\vin{h}} B^1_h - B^2_h \xi_{\vout{h}})_h
  \oplus (\xi_i I^1_i)_i \oplus (-J^2_i\xi_i)_i,
  \\
  \beta(C,D,E) = \sum_{\substack{h\in Q_1^{\rm dbl} \\ \vin{h} = i}}
  \varepsilon(h) (B^2_h C_{\overline{h}} + C_h B^1_{\overline{h}})
  + I^2_i E_i + D_i J^1_i.
\end{gathered}
\end{equation}
This complex appears at various points in study of quiver varieties,
such as
\begin{enumerate}
      \item the construction of instantons on an ALE space
    \cite[(4.3)]{Nakajima:KN},
      \item the tautological homomorphism in the definition of
    Kashiwara crystal structure on the set of irreducible components
    of lagrangian subvarieties \cite[\S4]{Nakajima:Na98},
      \item the definition of the Hecke correspondence
    \cite[\S5]{Nakajima:Na98},
      \item the decomposition of the diagonal
    \cite[\S6]{Nakajima:Na98},
      \item the definition of tensor product varieties
    \cite[\S3]{Nakajima:Na01}.
\end{enumerate}

\section{GIT quotients}

Since the group $G_V$ is noncompact, the quotient topology on
$\mu^{-1}(0)/G_V$ is \emph{not} Hausdorff in general. The trouble is
caused by nonclosed $G_V$-orbits: If orbits $O_1$, $O_2$ intersect in
their closure $\overline{O_1}\cap\overline{O_2}$, the corresponding
points in $\mu^{-1}(0)/G_V$ cannot be separated by disjoint open
neighborhoods.

\subsection{Affine quotients}
One solution to this problem is to introduce a coarser equivalence
relation
\begin{equation*}
  x \sim y \Longleftrightarrow
  \overline{G_V x}\cap \overline{G_V y}\neq \emptyset.
\end{equation*}
Then the quotient space $\mu^{-1}(0)/\!\!\sim$ is a Hausdorff
space. Let us denote this space by $\mu^{-1}(0)\dslash G_V$. It is
known that it has a structure of an affine algebraic scheme, in fact
we have
\begin{equation*}
   \mu^{-1}(0)\dslash G_V = \operatorname{Spec} \CC[\mu^{-1}(0)]^{G_V},
\end{equation*}
where $\CC[\mu^{-1}(0)]$ is the coordinate ring of the affine scheme
$\mu^{-1}(0)$, and $\CC[\mu^{-1}(0)]^{G_V}$ is its $G_V$-invariant
part. It is a fundamental theorem (due to Nagata) in geometric
invariant theory that the invariant ring is finitely generated. The
space $\mu^{-1}(0)\dslash G_V$ is called the \emph{affine
  algebro-geometric quotient} of $\mu^{-1}(0)$ by $G_V$. Let us denote
it by $\fM_0(V,W)$.

The ring of invariants is generated by two types of functions: (1)
Take an oriented cycle in the doubled quiver $Q^{\rm dbl}$ and
consider the trace of the composition of corresponding linear
maps. (2) Take a path starting from $i$ to $j$ and consider the
composition of $I_i$, linear maps for edges in the path, and $J_j$, a
linear map $W_i\to W_j$. Then its entry is an invariant function. This
follows from \cite{Nakajima:MR958897} after Crawley-Boevey's trick.

When $W=0$, $\fM_0(V,0)$ parametrizes semisimple representation of the
preprojective algebra $\Pi(Q)$. Roughly it is proved as
follows. Suppose a representation $B$ has a subrepresentation $B'$. We
have a short exact sequence $0\to B'\to B \to B/B'\to 0$. By the
action of `triangular' elements in $G_V$, we can send the off-diagonal
entries to $0$, in other words, $B$ can be degenerated to the direct
sum $B'\oplus (B/B')$. But $\fM_0(V,0)$ parametrizes closed orbits,
hence $B\cong B'\oplus (B/B')$. It means that we can take
complementary subrepresentation of $B'$. We continue this process
until it becomes a direct sum of simple representations.

The same is true even for $W\neq 0$ if we understand \emph{semisimple}
representations appropriately.

We can consider similar quotients of $\mathbf M(V)$ or $\mathbf
N(V)$. But they are often simple spaces:

\begin{ex}
    (1) Consider a quiver $Q$ without oriented cycles (e.g., a finite
    $ADE$ quiver) and define the affine algebro-geometric quotient
    $\mathbf N(V)\dslash G_V$ as above. Since we do not have oriented
    cycles, there is no invariant function. Hence
    $\mathbf N(V)\dslash G_V$ consists of a single point $\{ 0\}$.

    (2) Let $V$ be an $n$-dimensional complex vector space. Consider a
    $\GL(V)$-action on $\End(V)$ given by conjugation. Then
    $\End(V)\dslash \GL(V)$ is identified with $\CC^n/\mathfrak S_n$,
    the space of eigenvalues up to permutation. Here $\mathfrak S_n$
    is the symmetric group of $n$ letters.
\end{ex}

With a little more effort, one show
\begin{Exercise}
  (1) Consider an $ADE$ quiver $Q$ and define the affine
  algebro-geometric quotient $\fM_0(V,0) = \mu^{-1}(0)\dslash G_V$ for
  $W=0$. Show that it consists of a single point $\{ 0\}$. (This can
  be deduced from Lusztig's result saying that $B\in\mu^{-1}(0)$ is
  always nilpotent for an $ADE$ quiver. Alternative proof is given in
  \cite[Prop.~6.7]{Nakajima:Na94}.)

  (2) Consider the Jordan quiver with an $n$-dimensional vector space
  $V$ and $W=0$. Then $\fM_0(V,0) = \mu^{-1}(0)\dslash \GL(V)$ is
  $(\CC^2)^n/\mathfrak S_n$, the space of pairs of eigenvalues of
  $B_1$, $B_2$ up to permutation.
\end{Exercise}

On the other hand $\fM_0(V,W)$ (in general) and $\fM_0(V,0)$ for non
$ADE$ quiver are quite often complicated spaces.

\begin{ex}\label{ex:A1}
  Consider the $A_1$ quiver with vector spaces $V$, $W$. Then
  \(
     \mathbf M(V,W) = \Hom(W,V)\oplus\Hom(V,W)
  \)
  with $\mu(I,J) = IJ$. We consider $A = JI\in\End(W)$. It is
  invariant under $\GL(V)$ and its entries are $\GL(V)$-invariant
  functions on $\mathbf M(V,W)$. A fundamental theorem of the
  invariant theory says that they generate the ring of invariants. It
  satisfies $A^2 = JIJI = 0$ if $\mu(I,J) = 0$. A little more effort shows
  \begin{equation*}
    \fM_0(V,W) = \{ A\in\End(W) \mid A^2 = 0, \operatorname{rank} A\le \dim V
    \}.
  \end{equation*}
\end{ex}
Note that $\mathbf N(V,W)\dslash \GL(V)$ is $\{0\}$ in this example.

\subsection{GIT quotients}
Another way to construct a nice quotient space is to take a
$G_V$-invariant open subset $U$ of $\mu^{-1}(0)$ so that arbitrary
$G_V$-orbit in $U$ is closed (in $U$). Such an open subset $U$ arises
in geometric invariant theory. Since it is not our intension to
explain detailed structures of the quotient as an algebraic variety,
let us directly goes to a definition of the open subset $U$. In fact,
it depends on a choice, the stability parameter
$\zeta^\RR = (\zeta^\RR_i)\in \RR^{Q_0}$.

We consider $\zeta^\RR$ as a function $\ZZ^{Q_0}\to \RR$ by
$\zeta^\RR((v_i)_i) = \sum \zeta^\RR_i v_i$.

\begin{defn}
We say $(B,I,J)\in\mathbf M(V,W)$ is \emph{$\zeta^\RR$-semistable} if
the following two conditions are satisfied:
\begin{enumerate}
\item If a $Q_0$-graded subspace $S = \bigoplus S_i$ in $V$ is
  contained in $\Ker J$ and $B$-invariant, then
  $\zeta^\RR({\dim} S)\le 0$.
\item If a $Q_0$-graded subspace $T = \bigoplus T_i$ in $V$ contains
  $\Ima I$ and $B$-invariant, then
  $\zeta^\RR(\dim T)\le \zeta^\RR(\dim V)$.
\end{enumerate}
We say $(B,I,J)$ is \emph{$\zeta^\RR$-stable} if the strict
inequalities hold in (1),(2) unless $S=0$, $T=V$ respectively.
\end{defn}

\begin{rem}
  In view of Remark~\ref{rem:CB}, we can express the condition in
  terms of $\mathbf M(V')$: we set
  $\zeta^\RR_\infty = -\zeta^\RR(\dim V)$. Then $\zeta^\RR$ is
  extended to a function $\ZZ^{Q_0\sqcup\{\infty\}}\to\RR$. Then the
  condition is equivalent to $\zeta^\RR(\dim S')\le 0$ for a graded
  invariant subspace $S'\subset V'$. According to either
  $S'_\infty = 0$ or $\CC$, we have the above two cases (1), (2)
  respectively. Note that $S'$ being invariant means $S'$ is a
  submodule. Hence this reformulation coincides with the standard
  King's stability condition \cite{Nakajima:King}.
\end{rem}

Let us give a simple consequence of the $\zeta^\RR$-stability
condition. It basically says a $\zeta^\RR$-stable framed
representation is Schur:
\begin{prop}
  Suppose $(B,I,J)$ is $\zeta^\RR$-stable. Then the kernel of $\iota$
  and cokernel of $d\mu$ in q \textup(the framed version of\textup)
  \eqref{eq:3} are trivial.
\end{prop}

\begin{proof}
  By \eqref{eq:1} it is enough to check the assertion for
  $\iota$. Suppose $\xi = (\xi_i)_i$ is in $\Ker\iota$. Then $\Ima\xi$
  is $B$-invariant and contained in $\Ker J$. Therefore
  $\zeta^\RR(\dim \Ima\xi) \le 0$ by the $\zeta^\RR$-semistability
  condition. Similarly $\Ker\xi$ is $B$-invariant and contains
  $\Ima I$. Therefore $\zeta^\RR(\dim \Ker\xi) \le \zeta^\RR(\dim
  V)$. But
  $\zeta^\RR(\dim \Ima\xi) + \zeta^\RR(\dim\Ker\xi) = \zeta^\RR(\dim
  V)$. Therefore two inequalities must be equalities. The
  $\zeta^\RR$-stability condition says $\Ima \xi = 0$ and $\Ker\xi =
  V$. These are nothing but $\xi = 0$.
\end{proof}

This also implies that the stabilizer of a stable point $(B,I,J)$ in
$G_V$ is trivial: If $g = (g_i)_i$ stablizes $(B,I,J)$,
$(g_i - \operatorname{id}_{V_i})_i$ is in the kernel of $\iota$,
hence must be trivial by the proposition.

The $G_V$-orbit through $(B,I,J)$ is of the form
$G_V/\operatorname{Stabilizer}$. Hence all $\zeta^\RR$-stable orbits
have the maximal dimension, equal to $\dim G_V$. Since an orbit $O_1$
appeared in the closure of an orbit $O_2$ has $\dim O_1 < \dim O_2$,
we conclude that $\zeta^\RR$-orbits are closed in the open subset of all
$\zeta^\RR$-stable points in $\mathbf M(V,W)$.

Let $\mu^{-1}(0)_{\zeta^\RR}^{\mathrm{s}}$ be the subset of
$\zeta^\RR$-stable points in $\mu^{-1}(0)$.

\begin{thm}
    \textup{(1)} $\mu^{-1}(0)_{\zeta^\RR}^{\mathrm{s}}$ is a complex
    manifold \textup(i.e., nonsingular\textup) whose dimension is
    $\dim\mathbf M(V,W) - \dim G_V$.

\textup{(2)} The quotient $\mu^{-1}(0)_{\zeta^\RR}^{\mathrm{s}}/G_V$ is a complex manifold whose dimension is $\dim\mathbf M(V,W) - 2\dim G_V$.
\end{thm}

The first assertion is a simple consequence of the inverse function
theorem, as $d\mu$ is surjective over
$\mu^{-1}(0)_{\zeta^\RR}^{\mathrm{s}}$.
The second assertion is a little more difficult to prove, but
it is a consequence of the assertion that the $G_V$-action on
$\mu^{-1}(0)_{\zeta^\RR}^{\mathrm{s}}$ is free and closed. For our
applications, this will be important as Poincar\'e duality isomorphism
holds for $\mu^{-1}(0)_{\zeta^\RR}^{\mathrm{s}}/G_V$.

It is known that the $\zeta^\RR$-semistability automatically implies
the $\zeta^\RR$-stability unless $\zeta^\RR$ lies in a finite union of
hyperplanes in $\RR^{Q_0}$. In this case, it is known that the natural
map
\begin{equation}\label{eq:pi}
  \pi\colon \mu^{-1}(0)_{\zeta^\RR}^{\mathrm{s}}/G_V \to \mu^{-1}(0)\dslash G_V
\end{equation}
is proper, i.e., inverse images of compact subsets remain
compact. Here the map is defined by assigning $\sim$-equivalence
classes to $\zeta^\RR$-stable $G_V$-orbits. If $(B,I,J)$ is regarded
as a framed representation of the preprojective algebra, it is sent to
its `semisimplification' under $\pi$.

When $\zeta^\RR$ lies in a finite union of hyperplanes, $\pi$ is not
proper. We need to replace $\mu^{-1}(0)_{\zeta^\RR}^{\mathrm{s}}/G_V$
by a larger space, a certain quotient of the space of
$\zeta^\RR$-semistable points in $\mu^{-1}(0)$, similar to
$\mu^{-1}(0)\dslash G_V$.
We do not consider such $\zeta^\RR$. We always assume
$\zeta^\RR$-stability and $\zeta^\RR$-semistability are equivalent
hereafter.

Let us denote $\mu^{-1}(0)_{\zeta^\RR}^{\mathrm{s}}/G_V$ by
$\fM_{\zeta^\RR}(V,W)$. We will use the case $\zeta^\RR_i > 0$ for all
$i\in Q_0$ later. In this case we simply denote it by $\fM(V,W)$. The
inverse image $\pi^{-1}(0)$ will be important. Let us denote it by
$\La(V,W)$.

For a $\zeta^\RR$-stable framed representation $(B,I,J)$, the
corresponding point in $\fM_{\zeta^\RR}(V,W)$ is denoted by $[B,I,J]$.

\begin{ex}\label{ex:A1st}
  Consider the $A_1$ quiver and vector spaces $V$, $W$ as in
  Example~\ref{ex:A1}. When the stability parameter $\zeta^\RR > 0$
  (resp.\ $\zeta^\RR < 0$), the $\zeta^\RR$-semistablity means that
  $J$ is injective (resp.\ $I$ is surjective). Note also that
  $\zeta^\RR$-semistability and $\zeta^\RR$-stability are
  equivalent. Suppose $\zeta^\RR > 0$ for brevity. Then $\Ima J$ is a
  subspace of $W$ with dimension $\dim V$. Hence we have a map
  $\fM(V,W)\to \operatorname{Gr}(V,W)$, the Grassmannian variety of
  subspaces in $W$ with dimension $\dim V$. In particular, we have
  $\fM(V,W) = \emptyset$ unless $0\le \dim V \le \dim W$.

  Consider $A = JI$ as in Example~\ref{ex:A1}. We have
  $\Ima A\subset\Ima J$ and $\Ima A \subset \Ker A$, hence
  $A\in\Hom(W/\Ima J, \Ima J)$. Moreover it is simple to check that
  $A$ together with $\Ima J$ conversely determines $(I,J)$ up to
  $\GL(V)$-action. This shows that
  $\fM(V,W)\cong T^*\operatorname{Gr}(V,W)$, the cotangent bundle of
  $\operatorname{Gr}(V,W)$.

  The map $\pi$ in \eqref{eq:pi} is given by $(\Ima J,A)\mapsto
  A$. Comparing with Example~\ref{ex:A1}, one see that $\pi$ is
  surjective when $\dim V \le \dim W/2$. In fact, it is known that it
  is a resolution of singularities. On the other hand, the image is a
  proper subset if $\dim V > \dim W/2$ as
  $\operatorname{rank} A \le \dim W - \dim V < \dim W/2$. Note also that
  $\La(V,W) = \operatorname{Gr}(V,W)$.
\end{ex}

In this example $\fM(V,W)$ is a cotangent bundle of
$\operatorname{Gr}(V,W)$ which is the quotient of
$\zeta^\RR$-semistable points in $\mathbf N(V,W)^* = \Hom(V,W)$ by
$\GL(V)$. But it is not the case as the following example illustrate:

\begin{ex}\label{ex:An}
    Consider the quiver of type $A_n$ with $\dim V_i = 1$ for all
    $i\in Q_0$, $\dim W_i = 1$ for $i=1$, $n$, $\dim W_i = 0$ for
    $i\neq 1,n$. Here we number vertices as usual.
    \begin{equation*}
        \xymatrix@C=1.2em{
          \CC \ar@<-.5ex>[rr]_{B_{2,1}} \ar@<.5ex>[d]^{J_1}
          && \CC \ar@<-.5ex>[ll]_{B_{1,2}} 
          \ar@<-.5ex>[rr]_{B_{3,2}} &&
          \CC \ar@<-.5ex>[ll]_{B_{2,3}} \ar@<-.5ex>[rr]_{B_{4,3}}
          && \ar@<-.5ex>[ll]_{B_{3,4}} \cdots
          \ar@<-.5ex>[rr]_{B_{n-1,n-2}} 
          && 
          \CC \ar@<-.5ex>[rr]_{B_{n,n-1}} \ar@<-.5ex>[ll]_{B_{n-2,n-1}} &&
          \CC \ar@<-.5ex>[ll]_{B_{n-1,n}} \ar@<.5ex>[d]^{J_n} 
          \\
          \CC \ar@<.5ex>[u]^{I_1} && && && && &&\CC \ar@<.5ex>[u]^{I_n}
        }
    \end{equation*}
    The ring of invariant functions is generated by
    \begin{equation*}
        x = J_n B_{n,n-1}\dots B_{2,1} I_1, \quad
        y = J_1 B_{1,2}\dots B_{n-1,n} I_n, \quad
        z = J_1 I_1,
    \end{equation*}
    which satisfies $xy = z^{n+1}$ thanks to equations
    $I_1 J_1 = B_{1,2} B_{2,1}$, etc (up to sign). Thus $\fM_0(V,W)$
    is the hypersurface $xy = z^{n+1}$ in $\CC^3$.

    Now we take the stability parameter $\zeta^\RR$ with
    $\zeta^\RR_i > 0$ for all $i$. Let us study
    $\pi\colon \fM(V,W) \to \fM_0(V,W)$. One first check that
    $(B,I,J)$ is $\zeta^\RR$-stable if it corresponds to
    $(x,y,z)\neq (0,0,0)$. In fact, there are no subspaces $S$, $T$
    appearing in the definition of $\zeta^\RR$-stability in this
    case. An interesting thing happens when $(x,y,z) = (0,0,0)$.
    Starting from $J_1 I_1 = 0$, we have $B_{i,i+1} B_{i+1,i} = 0$
    $i=1,\dots,n-1$, and $I_n J_n = 0$ thanks to $\mu = 0$. Since all
    vector spaces have dimension $1$, at least one of paired linear
    maps is zero. On the other hand, the $\zeta^\RR$-stability
    condition means that it is not possible that two linear maps
    starting from $V_i$ ($1\le i\le n$) cannot be simultaneously zero,
    as $V_i$ violates the condition then. Then one check that the only
    possibility is
    \begin{equation*}
        \xymatrix@C=1.2em{
          \CC \ar[d]^{J_1}
          && \CC \ar[ll]_{B_{1,2}} && \ar[ll]_{B_{2,3}} \cdots
          && \CC \ar[ll]_{B_{i-1,i}} \ar[rr]^{B_{i+1,i}} && \cdots
          \ar[rr]^{B_{n-1,n-2}}
          && \CC \ar[rr]^{B_{n,n-1}} &&
          \CC \ar[d]^{J_n} 
          \\
          \CC && && && && && && \CC 
        }
    \end{equation*}
    for some $i = 1,\dots, n$. Here only nonzero maps are written. By
    the $G_V$-action, we can normalize all maps as $1$ except
    $B_{i+1,i}$, $B_{i-1,i}$. Then the remaining data is
    $(B_{i+1,i}, B_{i-1,i})\in\CC^2\setminus 0$ modulo the action of
    $\GL(V_i) = \CC^\times$. We thus get the complex projective line
    $\mathbb CP^1$. Let us denote this by $\fC_i$. Thus we have
    $\La(V,W) = \fC_1\cup \fC_2\cup \dots \cup\fC_n$. The intersection
    $\fC_i\cap \fC_{i+1}$ is $B_{i+1,i} = 0$, hence is a single
    point. Other intersection $\fC_i\cap \fC_j$ is empty. Thus
    $\fC_i$'s form a chain of $n$ projective lines.
    In fact, $\fM(V,W)$ is the minimal resolution of the simple
    singularity $xy=z^n$ of type $A_n$.
\end{ex}

This example can be generalized to other $ADE$ singularities as
follows. Take an affine Dynkin diagram of type $ADE$ and consider the
primitive (positive) imaginary root vector $\delta$. We remove the
special vertex $0$ and take the corresponding vector space $V$ of the
finite $ADE$ quiver. The entry of the special vertex $0$ is always
$1$, and let us make it to $W$ for the finite $ADE$ quiver. For type
$A_n$, the vertex $0$ is connected to $1$ and $n$, hence we set
$W_1 = \CC$, $W_n = \CC$ and $W_i = 0$ otherwise as the above
example. For other types, $0$ is connected to a single vertex, say
$i_0$. Hence we take $W_{i_0} = \CC$ and $W_i = 0$ otherwise. Then
$\fM_0(V,W)$ is the simple singularity of the corresponding type, and
$\fM(V,W)$ is its minimal resolution. This is nothing but Kronheimer's construction \cite{Nakajima:Kr}.

The exceptional set of the minimal resolution, i.e., the inverse image
of $0$ under $\pi$ (which is our $\La(V,W)$) is known to be union of
projective lines intersecting as the Dynkin diagram. Let us check this
assertion for $D_4$.

\begin{ex}\label{ex:D4}
    We consider $\fM(V,W)$, $\fM_0(V,W)$ of type $D_4$ with
    
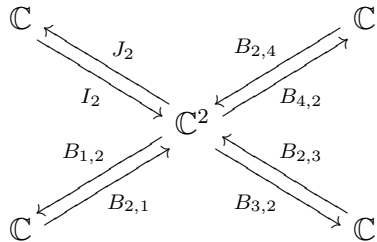
\begin{equation*}
        \xymatrix@C=4em{
          \CC \ar@<-.5ex>[dr]_{I_2} && \CC 
          \ar@<-.5ex>[dl]_{B_{2,4}}
          \\
          & \CC^2 \ar@<-.5ex>[ul]_{J_2}
          \ar@<-.5ex>[dl]_{B_{1,2}}
          \ar@<-.5ex>[dr]_{B_{3,2}}
          \ar@<-.5ex>[ur]_{B_{4,2}}
          && 
          \\
          \CC \ar@<-.5ex>[ur]_{B_{2,1}}
          && \CC \ar@<-.5ex>[ul]_{B_{2,3}}
        }
    \end{equation*}
    where the upper left vector space is $W_2$ and others are
    $V_i$'s. As is observed in Example~\ref{ex:An}, it is helpful to
    consider a vector subspace where there are no coming linear
    maps. Suppose $V_1$ (the left lower space) is so, i.e.,
    $B_{1,2} = 0$. Then the data with $V_1$ removed, i.e.,
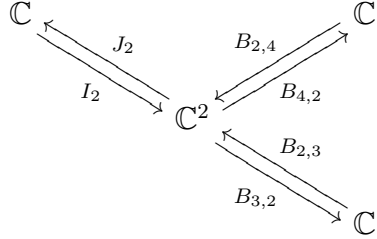
    \begin{equation*}
        \xymatrix@C=4em{
          \CC \ar@<-.5ex>[dr]_{I_2} && \CC 
          \ar@<-.5ex>[dl]_{B_{2,4}}
          \\
          & \CC^2 \ar@<-.5ex>[ul]_{J_2}
          \ar@<-.5ex>[dr]_{B_{3,2}}
          \ar@<-.5ex>[ur]_{B_{4,2}}
          && 
          \\
          && \CC \ar@<-.5ex>[ul]_{B_{2,3}}
        }
    \end{equation*}
    is also $\zeta$-stable. It is easy to check that the corresponding
    space $\fM(V',W)$ ($V' = V\ominus V_1$) is a single
    point. Conversely we start from the point $\fM(V',W)$ and add
    $B_{2,1}$ to get a point in $\fM(V,W)$. As in the type $A_n$ case,
    points constructed in this way form the complex line. Let us
    denote it by $\fC_1$. Replacing $V_1$ by $V_3$, $V_4$, we have
    $\fC_3$, $\fC_4$.

    Let us focus on $V_2$. Contrary to other vertices, we cannot
    remove the whole $V_2$, as it violates the
    $\zeta^\RR$-stability. We instead replace $V_2$ by one dimensional
    space $V_2$. Then all vector spaces are $1$-dimensional, and it is
    easy to check that the corresponding variety $\fM(V',W)$ is a
    single point given by
    
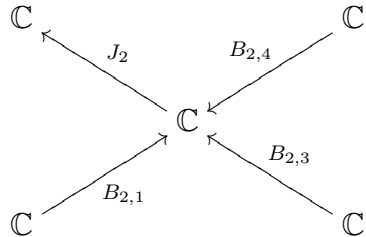
\begin{equation*}
        \xymatrix@C=4em{
          \CC 
          && \CC  \ar[dl]_{B_{2,4}}
          \\
          & \CC \ar[ul]_{J_2}
          && 
          \\
          \CC \ar[ur]_{B_{2,1}}
          && \CC \ar[ul]_{B_{2,3}}
        }
    \end{equation*}
    All written maps are nonzero. When we add one dimensional vector
    space to $V'_2$, we consider it as a subspace in
    $V_1\oplus V_3\oplus V_4$ by $B_{1,2}$, $B_{3,2}$,
    $B_{4,2}$. Since $\mu=0$ is satisfied, it must be contained in the
    kernel of
    $(B_{2,1}, B_{2,3}, B_{2,4})\colon V_1\oplus V_3\oplus V_4\to
    V'_2$,
    which is a $2$-dimensional space. Therefore points constructed in
    this way also form the complex line, denoted by $\fC_2$. It is also
    clear that $\fC_2$ meets with $\fC_1$, $\fC_3$, $\fC_4$ at three distinct
    points, hence the configuration forms the Dynkin diagram $D_4$.
\end{ex}

Subvarieties $\fC_i$ are examples of Hecke correspondence, defined in
\S\ref{subsec:Hecke}, where the factor $\fM(V',W)$ is a single point
as we have seen above, hence is a subvariety in $\fM(V,W)$. It will be
also clear that why $\fC_i$ is a projective space : it is a projective
space associated with a certain $\operatorname{Ext}^1$.

\section{Representations of Kac-Moody Lie algebras}

In this section we assume that the quiver $Q$ has no edge
loops. Therefore we have the (symmetric) Kac-Moody Lie algebra
$\mathfrak g = \mathfrak g_Q$ whose Dynkin diagram is the graph
obtained from $Q$ by replacing oriented allows by unoriented edges. If
$Q$ is of type $ADE$, the Kac-Moody Lie algebra is a complex simple
Lie algebra of the corresponding type.

\begin{rem}
  When $Q$ has an edge loop, say the Jordan quiver, it was not a
  priori clear what is an analog of $\mathfrak g_Q$. Recently
  Maulik-Okounkov find a definition of a Lie algebra based on quiver
  varieties possibly with edge loops \cite{Nakajima:MO}. Bozec also
  studies a generalized crystal structure on the set of irreducible
  components \cite{Nakajima:Bo}.
\end{rem}

\subsection{Lagrangian subvariety}

\begin{thm}[\protect{\cite[Th.~5.8]{Nakajima:Na94}}]
    $\La(V,W)$ is a lagrangian subvariety in $\fM(V,W)$. In particular, all irreducible components of $\La(V,W)$ has dimension $\dim\fM(V,W)/2$.
\end{thm}

The proof is geometric, hence is omitted. We at least see that it is
true for above examples. One can also check that it is not true for
Jordan quiver with $\dim V = \dim W = 1$, as $\fM(V,W) = \CC^2$,
$\La(V,W) = \{0\}$. Thus it is important to assume that $Q$ has no
edge loops.

We consider the top degree homology group of $\La(V,W)$
\begin{equation*}
    H_{d(V,W)}(\La(V,W)),
\end{equation*}
where $d(V,W) = \dim_\CC\fM(V,W) = \dim\mathbf M(V,W) - 2\dim G_V$. It
should not be confused with cohomology groups of modules. $\La(V,W)$
is a topological space with classical topology, and we consider its
singular homology group.
Here we consider homology groups with \emph{complex} coefficients,
though we can consider \emph{integer} coefficients also. It is known
that $\La(V,W)$ is a lagrangian subvariety in $\fM(V,W)$ with respect
to the symplectic structure explained in the previous section. In
particular, its dimension is half of $d(V,W)$, hence the above is the
top degree homology. Thus $H_{d(V,W)}(\La(V,W))$ has a base given by
irreducible components of $\La(V,W)$.

A reader who is not comfortable with homology groups could use the
space of constructible functions on $\La(V,W)$ instead. The definition
of the action is in parallel, though the construction of a base
corresponding to irreducible components is more involved. The
construction of the base is due to Lusztig, and is called
\emph{semicanonical base}.

\subsection{Examples}

Our main goal in this section is to explain that the direct sum\linebreak
$\bigoplus_V H_{d(V,W)}(\La(V,W))$ has a structure of an integrable
representation of $\mathfrak g$ with highest weight $\dim W$. Let us
first check it in the level of dimension (or weights).

Take $A_1$ as in Example~\ref{ex:A1st}. We have
\begin{equation*}
  \dim H_{d(V,W)}(\La(V,W)) =
  \begin{cases}
    1 & \text{if $0\le \dim V\le\dim W$},\\
    0 & \text{otherwise}.
  \end{cases}
\end{equation*}
This is the same as weight spaces of the finite dimensional
irreducible representation of $\mathfrak g = \mathfrak{sl}(2)$ with
highest weight $n = \dim W$.

Since this is a review for the proceeding of Symposium on Ring Theory
and Representation Theory, let us review the usual construction of
this representation. It is realized as the space of degree $n$
homogeneous polynomials in two variables:
\begin{equation*}
  \CC x^n\oplus \CC x^{n-1}y \oplus \dots\oplus \CC xy^{n-1}\oplus \CC y^n.
\end{equation*}
Here the $\mathfrak{sl}(2)$-action is induced from that on
$\operatorname{Span}(x,y) = \CC^2$. More concretely let us take a
standard base of $\mathfrak{sl}(2)$ as
\begin{equation}\label{eq:6}
   H =
   \begin{pmatrix}
     1 & 0 \\ 0 & -1
   \end{pmatrix}, \quad
   E =
   \begin{pmatrix}
     0 & 1 \\ 0 & 0
   \end{pmatrix}, \quad
   F =
   \begin{pmatrix}
     0 & 0 \\ 1 & 0
   \end{pmatrix}.
\end{equation}
Then
\begin{equation*}
  H x = x, \quad H y = -y, \quad
  E x = 0, \quad E y = x, \quad F x = y, \quad F y = 0.
\end{equation*}
The induced action means that $H$, $E$, $F$ acts on homogeneous
polynomials as derivation, for example
\begin{equation*}
  H x^n = n x^{n-1} Hx = nx^n,\quad E x^n = n x^{n-1} Ex = 0, \quad
  F x^n = n x^{n-1} Fx = nx^{n-1}y,\qquad \text{etc}.
\end{equation*}

Observe that $x^n$, $x^{n-1}y$, \dots, $y^n$ are eigenvectors of $H$
with eigenvalues $n$, $n-2$, \dots, $-n$. In this example, weight
spaces are all $1$-dimensional, and are scalar multiplies of those
vectors. (We have $(n+1)$ eigenvectors in total, and the total
dimension of the representation is $(n+1)$.)

Thus we see that dimension of weight spaces matches with dimension of
homology groups above. At this stage it looks just a coincidence.

Next consider Example~\ref{ex:An}. From we saw there, we have
\begin{equation*}
  H_{d(V,W)}(\La(V,W)) = H_{2}(\fC_1\cup\fC_2\cup\cdots\fC_n)
  = 
  \CC[\fC_1]\oplus \CC[\fC_2]\oplus\dots\oplus\CC[\fC_n],
\end{equation*}
where $[\ ]$ denotes the fundamental class. In this example, the Lie
algebra $\mathfrak g$ is $\mathfrak{sl}(n+1)$, and the representation
has the highest weight $\dim W = \varpi_1 + \varpi_n$, in other words,
it is the adjoint representation, i.e., the Lie algebra itself
considered as a representation with the action given by the Lie
bracket. Since we choose a particular $V$ (unlike $A_1$ example
above), the homology group corresponds to a weight space. In this
example, we consider the zero weight space, which is the space of
diagonal matrices in $\mathfrak{sl}(n+1)$. It is indeed
$n$-dimensional.

Let us again spell out the weight spaces of the adjoint representation
concretely. $\mathfrak{sl}(n+1)$ is the space of trace-free $(n+1)\times(n+1)$ complex matrices, regarded as a Lie algebra by the bracket $[A,B] = AB - BA$.
Let us denote by $\mathfrak h$ the space of diagonal matrices in
$\mathfrak{sl}(n+1)$. It forms a commutative Lie subalgebra in
$\mathfrak{sl}(n+1)$, and called a Cartan subalgebra. We have vector
space decomposition
\begin{equation*}
  \mathfrak{sl}(n+1) =
  \mathfrak h \oplus \bigoplus_{i\neq j} \CC E_{ij},
\end{equation*}
where $E_{ij}$ is the matrix unit for the entry $(i,j)$. This is the
simultaneous eigenspace decomposition of $\mathfrak{sl}(n+1)$ with
respect to the action of elements in $\mathfrak h$. The space
$\mathfrak h$ itself is the zero eigenspace, and $E_{ij}$ is an
eigenvector.

\begin{Exercise}
  Let $W$ be the same as above, but consider $\fM(V,W)$ for different
  $V$. Show that $\fM(V,W)$ and $\La(V,W)$ are either empty set or a
  single point. Check that it coincides with the weight spaces of the
  adjoint representation of $\mathfrak{sl}(n+1)$. (Recall the homology
  group of the empty set is $0$-dimensional vector space.) For
  example, if we remove $\CC$ at the $i$th vertex, it corresponds to
  $\CC E_{i,i+1}$.
\end{Exercise}

This is easy if all $V_i$ are at most $1$-dimensional (corresponding
to the matrix unit $E_{ij}$ with $i < j$). But one needs to use the
stability condition in an essential way if some $V_i$ has dimension
greater than $1$.

One could also show that Example~\ref{ex:D4} corresponds to the
adjoint representation of $\mathfrak g = \mathfrak{so}(8)$ so that
$H_{d(V,W)}(\La(V,W))$ is the space of diagonal matrices, and spaces
for other $V$ are either $0$ or $1$-dimensional. But this becomes even
more tedious calculation and the author never check it by mere
analysis without using general structure theory expained below.

\subsection{Convolution product}

As we write above already, we use homology groups to give a geometric
realization of representations of Kac-Moody Lie algebras.
A reader who prefers constructible functions skip this subsection and
goes to the next.

For homology groups, it is technically simpler to work with a version
of the Borel-Moore homology group, which turns out to be isomorphic to
the usual homology group for $\La(V,W)$.\footnote{This is because
  $\La(V,W)$ is a complex projective variety, hence a finite CW
  complex.}
A review of the definition of the Borel-Moore homology group and its
fundamental properties is found in \cite[App.~B]{Nakajima:Fu}. In our
situation, $\La(V,W)$ is a closed subspace in a smooth oriented
manifold $\fM(V,W)$ of real dimension $2d(V,W)$. Then the Borel-Moore
homology group is defined as
\begin{equation*}
  \begin{aligned}[t]
    H_*(\La(V,W)) & = H^{2d(V,W)-*}(\fM(V,W), \fM(V,W)\setminus\La(V,W))
    \\
    & = H^{2d(V,W)-*}(\fM(V,W), \La(V,W)^c).
  \end{aligned}
\end{equation*}
In fact, this definition makes sense for any embedding of $\La(V,W)$
into a smooth manifold, and is independent of the choice.

Let us take another $Q_0$-graded vector space $V'$ and consider
varieties $\La(V',W)$ also. Let us consider the fiber product
$Z(V,V',W)$,
\begin{equation*}
  Z(V,V',W) = \fM(V,W)\times_{\fM_0(V\oplus V',W)} \fM(V',W),
\end{equation*}
where $\fM(V,W) \text{(resp.\ $\fM(V',W)$)} \to\fM_0(V\oplus V',W)$ is
the composite of
$\pi\colon \fM(V,W) \linebreak[3]\text{(resp.\ $\fM(V',W)$)} \to
\fM_0(V,W) \text{(resp.\ $\fM_0(V',W)$)}$ and closed embeddings
$\fM_0(V,W) \linebreak[4] \text{(resp.\ $\fM(V',W)$)}\to \fM_0(V\oplus
V',W)$ is given by setting data for $B$ in $V'$ (resp.\ $V$) by
$0$. It is a closed subvariety in $\fM(V,W)\times\fM(V',W)$, and
called \emph{an analog of Steinberg variety} or \emph{a Steinberg-type
  variety}, as a similar space is considered by Steinberg for the case
of the cotangent bundle of a flag variety.
Note that the restriction of projection
$p_1, p_2 \colon Z(V,V',W)\to \fM(V,W), \fM(V',W)$ are proper (i.e.,
inverse images of compact subsets are compact) and
$p_2(p_1^{-1}(\La(V,W)))\subset \La(V',W)$.

We define the Borel-Moore homology group of $Z(V,V',W)$ as above,
using $\fM(V,W)\times\fM(V',W)$. Suppose $c\in H_k(Z(V,V',W))$. Then
we define the convolution product with $c$ by
\begin{equation*}
  c\ast \alpha = p_{2*}(c \cap p_1^*(\alpha)), \qquad
  \alpha\in H_{k'}(\La(V,W)).
\end{equation*}
Let us check that this is well-defined step by step. First $\alpha$
is in $H^{2d(V,W)-k}(\fM(V,W),\La(V,W)^c)$ as above.
Then its pull-back $p_1^*(\alpha)$ is
$H^{2d(V,W)-k}(\fM(V,W)\times\fM(V',W),p_1^{-1}(\La(V,W))^c)$. Its
intersection $c\cap p_1^*(\alpha)$ with $c$ is a class in
\begin{equation*}
  H^{4d(V,W)+2d(V',W)-k-k'}(\fM(V,W)\times\fM(V',W),
  (p_1^{-1}(\La(V,W))\cap Z(V,V',W))^c),
\end{equation*}
as we consider $c$ as an element of
$H^{2d(V,W)+2d(V',W)-k'}(\fM(V,W)\times\fM(V',W),
Z(V,V',W)^c)$.
Hence $c\cap p_1^*(\alpha)$ is a class in the Borel-Moore homology group
$H_{k+k'-2d(V,W)}(p_1^{-1}(\La(V,W))\cap Z(V,V',W))$. By our assumption
$p_1^{-1}(\La(V,W))\cap Z(V,V',W)$ is a compact set, hence the pushforward
homomorphism $p_{2*}\colon H_*(p_1^{-1}(\La(V,W))\cap Z(V,V',W))\to
H_*(\La(V',W))$ is defined. (See \cite[\S B2]{Nakajima:Fu}.)

From the computation of degrees, if $\alpha\in H_{d(V,W)}(\La(V,W))$,
then $c\ast \alpha\in H_{k - d(V,W)}(\La(V',W))$. Therefore if the
degree $k$ of $c$ is $d(V,W)+d(V',W)$, the degree of $c\ast \alpha$ is
$d(V',W)$. Note that $d(V,W)+d(V',W)$ is the complex dimension of
$\fM(V,W)\times \fM(V',W)$, hence the degree of $c$ is
$d(V,W)+d(V',W)$ means that it is a half-dimensional class in
$\fM(V,W)\times \fM(V',W)$. It is known that $Z(V,V',W)$ is lagrangian
for type $ADE$ (see \cite[Th.~7.2]{Nakajima:Na98}). Hence fundamental
classes of irreducible components of $Z(V,V',W)$ are examples of
half-dimensional cycles. 

\begin{ex}
  Consider the diagonal $\Delta\fM(V,W)$ in
  $\fM(V,W)\times\fM(V,W)$. Its fundamental class gives an operator
  $\Delta\fM(V,W)\ast\bullet\colon H_{d(V,W)}(\La(V,W))\to
  H_{d(V,W)}(\La(V,W))$ by the above cconstruction . It is the
  identity operator.
\end{ex}

\subsection{Hecke correspondence}\label{subsec:Hecke}

Fix $i\in Q_0$ and consider a pair $V'$, $V = V'\oplus S_i$ of
$Q_0$-graded spaces, where $S_i$ is $1$-dimensional at $i$ and $0$ at
other vertices. We define
$\mathfrak P_i(V,W)\subset\fM(V',W)\times\fM(V,W)$ consisting of
points $([B',I',J'], [B,I,J])$ such that $[B',I',J']$ is a framed
submodule of $[B,I,J]$ (\cite[\S5]{Nakajima:Na98}). More precisely, it
means that there is an injective linear map $\xi\colon V'\to V$ such
that $B\xi = \xi B'$, $I\xi = I'$, $J = J'\xi$. Thus we have a short
exact sequence of framed representations
\begin{equation}\label{eq:10}
   0 \to (B',I',J') \xrightarrow{\xi} (B,I,J) \to S_i \to 0,
\end{equation}
where $S_i$ is now regarded as a (simple) module with all linear maps
are $0$.

Let us explain the definition of operators for spaces of constructible
functions. We have two projections
$p_1, p_2\colon\mathfrak P_i(V,W)\to \fM(V',W), \fM(V,W)$. If $f$ is a
constructible function on $\La(V',W)$, we pull back it to
$\mathfrak P_i(V,W)\cap p_1^{-1}(\La(V',W))$ as
$p_1^* f = f\circ p_1$. Then we define its pushforward
$p_{2!} (p_1^* f)$ defined by
\begin{equation*}
  \left(p_{2!} (p_1^*f)\right) (x)
  = \sum_{a\in\CC} a \chi(p_2^{-1}(x)\cap (p_1^*f)^{-1}(a)),
\end{equation*}
where $\chi(\ )$ is the topological Euler number. This definition
corresponds to \eqref{eq:2} and we exchange roles of $p_1$, $p_2$ for
\eqref{eq:7}.

Let us explain the definition for homology groups. It was shown that
$\mathfrak P_i(V,W)$ is a smooth half-dimensional closed subvariety in
$\fM(V',W)\times\fM(V,W)$. By its definition, it is contained in
$Z(V',V,W)$. Thus the fundamental class $[\mathfrak P_i(V,W)]$ defines
an operator
\begin{equation}\label{eq:2}
  [\mathfrak P_i(V,W)]\ast\bullet\colon
  H_{d(V',W)}(\mathfrak L(V',W)) \to H_{d(V,W)}(\mathfrak L(V,W)).
\end{equation}
Changing the role of $\fM(V,W)$, $\fM(V',W)$, we also have
\begin{equation}\label{eq:7}
  [\mathfrak P_i(V,W)]\ast\bullet\colon
  H_{d(V,W)}(\mathfrak L(V,W)) \to H_{d(V',W)}(\mathfrak L(V',W)).
\end{equation}

\subsection{Definition of Kac-Moody action}\label{subsec:definition-kac-moody}

Like \eqref{eq:6} for $\mathfrak{sl}(2)$, a complex simple Lie algebra
has a presentation given by generators $E_i$, $F_i$, $H_i$ with
certain relations. For an example, generators for $\mathfrak{sl}(n+1)$
are $E_i = E_{i,i+1}$, $F_i = E_{i+1,i}$,
$H_i = E_{ii} - E_{i+1,i+1}$, where $E_{ij}$ is the matrix unit as
before. 
For a Kac-Moody Lie algera $\mathfrak{g}$, one needs to consider
Cartan subalgebra $\mathfrak{h}$, which is larger than
$\operatorname{Span} \{ H_i \}$. This is because we want to $H_i$ to
be linearly independent, even when the Cartan matrix has kernel. But
this is basically just convention and is not so important.
Let us ignore this difference, and defines action of $E_i$, $F_i$,
$H_i$ on the direct sum $\bigoplus_V H_{d(V,W)}(\La(V,W))$.

Let
\begin{equation}\label{eq:9}
  \begin{gathered}[m]
    F_i = \eqref{eq:2}, \qquad
    E_i = (-1)^{(d(V',W) - d(V,W))/2} \times \eqref{eq:7},
    \\
    H_i = (\dim W_i - \sum_{j} a_{ij} \dim V_j)
    \operatorname{id}_{H_{d(V,W)}(\La(V,W)},
  \end{gathered}
\end{equation}
where $a_{ij}$ is the Cartan matrix, i.e.,
$2\delta_{ij} - \# \{h\in Q_1^{\rm dbl}\mid \vout{h} = i, \vin{h} =
j\}$.
Note that $(\dim W_i - \sum_{j} a_{ij} \dim V_j)$ is the Euler
characteristic of the complex \eqref{eq:8} for
$(B^1,I^1,J^1) = (B,I,J)$, $(B^2, I^2,J^2) = S_i$, i.e.,
$(V^2,W^2) = (S_i,0)$ with linear maps $(B^2,I^2,J^2) = 0$. This is a
simple observation, and its brief explanation will be given below. It
is even more important to consider \eqref{eq:8} when one consider
larger algebras action on homology/K-theory of quiver varieties.

\begin{thm}[\cite{Nakajima:Na94,Nakajima:Na98}]\label{thm:KM}
  Operators \eqref{eq:9} satisfy the defining relations of the
  Kac-Moody Lie algebra $\mathfrak g$. Hence
  $\bigoplus_V H_{d(V,W)}(\La(V,W))$ is a representation of
  $\mathfrak g$. Moreover it is an \textup(irreducible\textup)
  integrable highest weight representation with the highest weight
  vector $[\fM(0,W)]\in H_0(\fM(0,W))$.
\end{thm}

When $V = 0$, the quiver variety $\fM(0,W)$ is a single point as all
linear maps $B$, $I$, $J$ are automatically $0$. As written above,
this is the highest weight vector with highest weight $\dim W$, i.e.,
it satisfies
\begin{equation*}
  \begin{split}
    & E_i [\fM(0,W)] = 0,\quad  H_i[\fM(0,W)] = \dim W_i [\fM(0,W)]
        \qquad \text{for all $i\in Q_0$},\\
    & \mathbf U(\mathfrak g)[\fM(0,W)] = \bigoplus_V H_{d(V,W)}(\fM(V,W)),
  \end{split}
\end{equation*}
where $\mathbf U(\mathfrak g)$ is the universal enveloping algebra of
$\mathfrak g$. The second condition, more concretely, means that the
direct sum $\bigoplus_V H_{d(V,W)}(\fM(V,W))$ is spanned by vectors
obtained from $[\fM(0,W)]$ by successively applying various $F_i$.

An integrability means that $E_i$, $F_i$ are locally nilpotent, that
is $E_i^N m = 0 = F_i^N m$ for sufficiently large $N = N(m)$ for a
vector $m$. (For a complex simple Lie algebra, it is known to be
equivalent to that the representation is finite dimensional.) It is
known that an integrable highest weight represenation is automatically
irreducible.

Let us briefly explain the proof of the first part of
Theorem~\ref{thm:KM}. The most delicate relation to check is
$[E_i, F_j] = \delta_{ij} H_i$. Once this is proved, the so-called
Serre relation follows from it together with the integrability.

It is relatively easy to check the relation for $i\neq j$. For the
proof of $[E_i, F_i] = H_i$, a key is to understand fibers of
projections $p_1, p_2\colon \mathfrak P(V,W)\to \fM(V',W), \fM(V,W)$.
By \eqref{eq:10}, the fiber of $p_2$ at $[B,I,J]$ is isomorphic to the
projective space associated with the vector space
$\Hom((B,I,J), S_i)$, where $\Hom$ is the space of homomorphism as
framed representations. This is the first cohomology of the complex
\eqref{eq:8} for $(B^1,I^1,J^1) = (B,I,J)$, $(B^2,I^2,J^2) = S_i$. As
we have remarked before, it is dual to the third cohomology of the
complex \eqref{eq:8} with $(B^1,I^1,J^1)$ and $(B^2,I^2,J^2)$ are
swapped.
On the other hand, the fiber of $p_1$ at $[B',I',J']$ is isomorphic to
the projective space associated with
$\operatorname{Ext}^1(S_i, (B',I',J'))$. This is the middle cohomology
of the complex \eqref{eq:8} for $(V^1,W^1) = (S_i, 0)$,
$(V^2,W^2) = (V',W)$. Then one observes that the complex \eqref{eq:8}
with $(V^1,W^1) = (S_i,0)$ has the vanishing first cohomology group if
$(B^2,I^2,J^2)$ satisfy the stability condition for $\zeta^\RR_i > 0$.
This is obvious as $0\neq \xi\in\Ker\alpha$ realizes $S_i$ as a
submodule of $(B,I,J)$. Then $\zeta^\RR(\dim S_i) = \zeta_i^\RR > 0$
violates the stability condition.
Thus the difference of dimensions of the second and third
cohomology groups of \eqref{eq:8} is the Euler characteristic of
\eqref{eq:8}, hence can be computed.
Now one uses that the Euler number of the complex projective space
$\mathbb CP^n$ is $n+1$ to complete the calculation.

\subsection{Inductive construction of irreducible components}
\label{subsec:crystal}

Let us sketch the proof of the second statement of
Theorem~\ref{thm:KM}. It is clear that $E_i^N m = 0$ for sufficiently
large $N$, as the dimension of $V_i$ cannot be negative. For
$F_i^N m = 0$, we use the vanishing of the first cohomology group of
\eqref{eq:8} for $V^1 = S_i$. If $N$ is sufficiently large, the
dimension of the first term exceeds that of the middle, hence $\alpha$
cannot be injective.

Let us next explain why the representation is highest weight. It means
that $\bigoplus H_{d(V,W)}(\La(V,W))$ is spanned by vectors obtained
from $[\fM(0,W)]$ by successively applying various $F_i$. This will be
shown by an inductive construction of irreducible components of
$\La(V,W)$. In fact, it also gives Kashiwara crystal structure on the
union of the set of irreducible components of $\La(V,W)$ with various
$V$. Since it is not our purpose to review crytal bases, we do not
explain this statement, and we concentrate only on the inductive
construction.

Let us take $[B,I,J]\in\fM(V,W)$ and consider \eqref{eq:8} with
$(B^1,I^1,J^1) = S_i$, $(B^2,I^2,J^2) = (B,I,J)$, i.e., 
\begin{equation*}
    V_i \xrightarrow{\alpha} \bigoplus_{j\neq i} V_j^{\oplus -a_{ij}} \oplus W_i
    \xrightarrow{\beta} V_i.
\end{equation*}
As we noted, the first cohomology group vanishes. Consider the third
cohomology group, which is the dual of the space $\Hom((B,I,J),S_i)$
of homomorphisms from $(B,I,J)$ to $S_i$. We have a natural
homomorphism
\(
    (B,I,J) \to \Hom((B,I,J),S_i)^\vee\otimes_\CC S_i,
\)
which is given by the natural projection
$V_i\to \operatorname{Cok}\beta$. In particular, it is surjective. We
consider the kernel of the natural homomorphism, and denote it by
$(B',I',J')$. Thus we have
\begin{equation}\label{eq:11}
    0\to (B',I',J')\to (B,I,J)\to
    \Hom((B,I,J),S_i)^\vee\otimes S_i\to 0.
\end{equation}
One can check that $(B',I',J')$ is $\zeta^\RR$-stable,
hence defines a point in $\fM(V',W)$ with
$\dim V' = \dim V - r\dim S_i$, where $r = \dim
\Hom((B,I,J),S_i)$. Moreover we have the induced exact sequence
\begin{equation*}
    \Hom((B,I,J),S_i)\otimes \Hom(S_i,S_i) \to
    \Hom((B,I,J),S_i) \to
    \Hom((B',I',J'),S_i) 
\end{equation*}
from the short exact sequence \eqref{eq:11}.
The first homomorphism is an isomorphism by the construction. The
second homomorphism is surjective as
$\operatorname{Ext}^1(S_i,S_i) = 0$. Therefore we conclude
$\Hom((B',I',J'),S_i) = 0$. It means that the complex
\begin{equation}\label{eq:12}
    V'_i \xrightarrow{\alpha'} \bigoplus_{j\neq i} V_j^{\prime\oplus -a_{ij}}
    \oplus W_i
    \xrightarrow{\beta'} V'_i
\end{equation}
has the vanishing third cohomology group.

Conversely we take $(B',I',J')$ with $\Hom((B',I',J'),S_i) = 0$. Then
we recover $(B,I,J)$ from an $r$-dimensional subspace in
$\operatorname{Ext}^1(S_i,(B',I',J')) = \Ker\beta'/\Ima\alpha'$.

We use this construction to understand $H_{d(V,W)}(\La(V,W))$ as
follows. (I learned this argument in \cite{Nakajima:Lu-crystal}.)
Let $Y$ be an irreducible component of $\La(V,W)$ with $V\neq 0$. We
define $\varepsilon_i(Y)$ be $\dim \Hom((B,I,J), S_i)$ for a generic
$[B,I,J]\in Y$. From the nilpotency of $(B,I,J)$, there exists
$i\in Q_0$ such that $\varepsilon_i(Y) > 0$. Set
$r = \varepsilon_i(Y)$. Then we
$Y^\circ = \{ [B,I,J]\in Y \mid \dim \Hom((B,I,J), S_i) = r \}$ is
open in $Y$. We apply the above construction to $[B,I,J]\in Y^\circ$
to obtain an irreducible variety $Y^{\prime\circ}$ in $\fM(V',W)$ with
$\dim V' = \dim V - r\dim S_i$. It can be shown that its closure
$Y' = \overline{Y^{\prime\circ}}$ is an irreducible component of
$\La(V',W)$. In fact, $Y'\subset\La(V',W)$ is clear from the
definition, as $(B,I,J)$ and $(B',I',J')$ have the same image under
$\pi$. Next note that
\begin{equation*}
    d(V,W) - d(V',W) = 2 r
    (\text{Euler characteristic of \eqref{eq:12}}
      - r).
\end{equation*}
On the other hand, $Y^\circ$ is the total space of Grassmann bundle of
$r$-planes in the vector bundle over $Y^{\prime\circ}$ with fiber
$\operatorname{Ext}^1(S_i,(B',I',J'))$. Hence its dimension is equal
to
$\dim Y^{\prime\circ} + r(\dim \operatorname{Ext}^1(S_i,(B',I',J')) -
r)$.
Since the Euler characteristic of \eqref{eq:12} is
$\dim \operatorname{Ext}^1(S_i,(B',I',J'))$, we conclude that
$\dim Y'$ is half-dimensional in $\fM(V',W)$.

We deduce
\begin{equation*}
    \frac{F_i^r}{r !} [Y'] = \pm [Y] + \sum_{\varepsilon_i(Y'') > r} c_{Y''} [Y'']
    \qquad c_{Y''}\in\mathbb Q.
\end{equation*}
By induction with respect to $\dim V$ and $\varepsilon_i$, we get the
assertion.

\begin{ex}
    Let us give an example of the induction of irreducible
    components. Let us consider the $A_2$-quiver with
    $\dim V = (1,2)$, $\dim W = (1,2)$. We have an irreducible
    component $Y$ with $\varepsilon_2(Y) = 1$, which is obtained from
    $\La(V',W)$ with $\dim V' = (1,0)$, which is a single
    point. Nonzero maps in $Y$ are
    \begin{equation*}
        \begin{CD}
            V_1 = \CC @<B_{1,2}<< V_2 = \CC^2 \\ @V{J_1}VV @VV{J_2}V \\
            W_1 = \CC @. W_2 = \CC^2.
        \end{CD}
    \end{equation*}
    We can normalize $J_1 = 1$ by $\GL(V_1)$, then we see that $Y$ is
    $\mathbb CP^2$ as $B_{1,2}\oplus J_2$ defines $2$-dimensional
    subspace in $V_1 \oplus W_2 = \CC^3$.

    Let us consider $\varepsilon_1(Y)$. For generic $[B,I,J]\in Y$, we
    have $B_{1,2}\neq 0$, hence $\varepsilon_1(Y) = 0$. We add
    $1$-dimensional space at the vertex $1$, and consider the
    irreducible component $Y''$ of $\La(V'',W)$ with
    $\dim V'' = (2,2)$. Over $[B,I,J]\in Y$, it is given by a
    $1$-dimensional subspace in the middle cohomology of the complex
    \begin{equation*}
        \CC = V_1 \xrightarrow{0\oplus J_1} V_2\oplus W_1 = \CC^2\oplus \CC
       \xrightarrow{(B_{1,2},0)} V_1 = \CC.
    \end{equation*}
    If $B_{1,2} \neq 0$, the middle cohomology is $1$-dimensional,
    hence the choice of a $1$-dimensional subspace is unique.
    But note that there is a point $B_{1,2} = 0$ in $Y$. Then the
    middle cohomology group is $2$-dimensional, hence we have choices
    parametrized by $\mathbb CP^1$. This shows that $Y''$ is the
    blowup of $Y = \mathbb CP^2$ at the point $B_{1,2} = 0$. It also
    gives an example where $\dim\Hom((B,I,J),S_i)$ jumps at a special
    point in an irreducible component.
\end{ex}


\ifx\undefined\bysame 
\newcommand{\bysame}{\leavevmode\hbox to3em{\hrulefill}\,} 
\fi

\end{document}